\documentclass{amsart}
\usepackage[utf8]{inputenc}
\usepackage[all]{xy}
\usepackage{upgreek,bbm,bm}
\usepackage{amssymb,amstext,amsmath,amscd,amsthm,hyperref,tikz,tikz-cd,amsfonts,enumerate,graphicx,latexsym,color,geometry}
\newtheorem{thm}{Theorem}[section]
\newtheorem{cor}[thm]{Corollary}
\newtheorem{lem}[thm]{Lemma}
\newtheorem{prop}[thm]{Proposition}

\newtheorem*{thm*}{Theorem}
\newtheorem*{cor*}{Corollary}
\theoremstyle{definition}

\newtheorem{rem}[thm]{Remark}
\newtheorem{ques}[thm]{Question}

\newtheorem*{conj*}{Conjecture}

\newtheorem*{claim*}{Claim}
\newtheorem*{ques*}{Question}
\newtheorem{nota*}[thm]{Notation}

\theoremstyle{remark}
\newtheorem*{ac}{Acknowledgments}
\newtheorem*{conv}{Convention}
\numberwithin{equation}{thm}

\def\Ass{\operatorname{Ass}}

\def\depth{\operatorname{depth}}

\def\Ext{\operatorname{Ext}}
\def\ge{\geqslant}

\def\grade{\operatorname{grade}}
\def\height{\operatorname{ht}}
\def\H{\mathrm{H}}
\def\Hom{\operatorname{Hom}}

\def\id{\operatorname{id}}
\def\im{\mathrm{Im}}
\def\le{\leqslant}
\def\m{\mathfrak{m}}
\def\Min{\operatorname{Min}}

\def\n{\mathfrak{n}}

\def\p{\mathfrak{p}}
\def\pd{\operatorname{pd}}
\def\q{\mathfrak{q}}
\def\rank{\operatorname{rank}}
\def\rhom{\operatorname{\mathbf{R}Hom}}
\def\spec{\operatorname{Spec}}
\def\Supp{\operatorname{Supp}}

\def\Tor{\operatorname{Tor}}
\def\type{\operatorname{type}}

\def\Z{\mathbb{Z}}
\begin{document}
\title{On injective dimension of the conormal module}
\author{Mohsen Gheibi}
\address[Mohsen Gheibi]{Department of Mathematics, Florida A{\&}M University, Tallahassee FL, USA }
\email{mohsen.gheibi@famu.edu}
\author{Kaito Kimura}
\address[Kaito Kimura]{Department of Mathematics, Purdue University, West Lafayette, IN, USA}
\email{m21018b@gmail.com}
\thanks{Kimura was partly supported by JSPS Overseas Research Fellowships}
\begin{abstract} Let $Q$ be a Gorenstein local ring, let $I$ be a proper ideal of $Q$, and set $R=Q/I$. We investigate when finite injective dimension of the conormal module $I/I^2$ forces $I$ to be generated by a regular sequence. We prove this implication for every ideal in the linkage class of a complete intersection. When $Q$ is regular, we also establish it under several numerical conditions in terms of the number of generators of $I$ and the type of $R$. For normal domain quotients, we show that finite injective dimension of $I/I^2$ forces $(\height(I)+1)[\omega_R]=0$ in $\operatorname{Cl}(R)$, yielding the desired conclusion when the divisor class group is torsionfree. In the graded setting, we further prove the implication for all monomial ideals in polynomial rings. We also obtain a converse to Kunz's description of the first Koszul homology of an almost complete intersection.
\end{abstract}
\maketitle
\section{Introduction}
The study of complete intersections through the homological properties of the conormal module has its origins in work of Ferrand and Vasconcelos \cite{F,V}. In 1967, they independently established that an ideal $I$ of finite projective dimension in a commutative Noetherian local ring $Q$ is generated by a regular sequence if and only if its conormal module $I/I^2$ is free over $Q/I$. In 1978, Vasconcelos \cite{V2} proposed that freeness could be weakened to finite projective dimension. His conjecture therefore asserted that
\[
\pd_Q I<\infty
\quad\text{and}\quad
\pd_{Q/I}(I/I^2)<\infty
\]
force $I$ to be generated by a regular sequence. Among the advances toward this conjecture was the work of Avramov and Herzog \cite{AH}, who established it in the graded characteristic-zero setting. The conjecture was ultimately proved in full generality by Briggs \cite{B}, using the homotopy Lie algebra associated to the quotient map.
Motivated by Briggs's theorem, we investigate whether finite injective dimension of the conormal module likewise detects the complete intersection property. More precisely, we consider the following question, which was raised in \cite[Question 5.18]{HMN} and was also considered in \cite[$\S$ 5]{A}.

\medskip
\noindent\textbf{Question.}
Let $Q$ be a regular local ring, let $I$ be a proper ideal of $Q$, and set $R=Q/I$. If $$\id_R(I/I^2)<\infty,$$ must $I$ be generated by a $Q$-regular sequence?
\medskip

The converse holds: if $I$ is generated by a regular sequence, then $R$ is Gorenstein and $I/I^2$ is free over $R$, hence has finite injective dimension. Moreover, for $I\ne0$, finite injective dimension of $I/I^2$ implies that $R$ is Cohen--Macaulay. If one can further show that $R$ is Gorenstein, then the conormal module has finite projective dimension, and Briggs's theorem applies. Thus, the Gorenstein property of the quotient provides a crucial link between the injective-dimension question and Vasconcelos' conjecture. 

We establish an affirmative answer for several classes of ideals. Our principal result concerns ideals in the linkage class of a complete intersection, commonly called \emph{licci} ideals. For these ideals, we can work over a Gorenstein ambient ring.

\medskip
\noindent\textbf{Theorem A.}
\emph{Let $Q$ be a Gorenstein local ring, let $I$ be a licci ideal of $Q$, and set $R=Q/I$. If
\[
\id_R(I/I^2)<\infty,
\]
then $I$ is generated by a $Q$-regular sequence.}
\medskip

Our next results give numerical conditions under which the injective-dimension question has an affirmative answer, without assuming that the ideal is licci.

\medskip
\noindent\textbf{Theorem B.}
\emph{Let $Q$ be a regular local ring, let $I$ be a proper ideal of $Q$, set $R=Q/I$, and suppose that $\id_R(I/I^2)<\infty$. Then $I$ is generated by a regular sequence if any of the following conditions holds:
\begin{enumerate}
\item $\mu(I)\le\height(I)+2$.
\item $R$ is generically Gorenstein and
$\mu(I)\le\height(I)+\type(R)$.
\item $\mu(I)$ is a prime number, and either $Q$ is unramified or $R$ is generically Gorenstein.
\end{enumerate} }
\medskip

The first assertion extends the almost complete intersection case proved in \cite[Theorem 4.8]{GT}. The arguments compare the number of generators and multiplicity of the conormal module with the  type of $R$, using the first Koszul homology to obtain the required bounds. We also apply these ideas to ideals of maximal minors of maximal height.

For normal domains, we obtain a further restriction in terms of the divisor class group. Finite injective dimension of the conormal module forces the canonical class to be torsion in the following precise sense.

\medskip
\noindent\textbf{Theorem C.} \emph{ Let $Q$ be a regular local ring, let $I$ be an ideal of $Q$ of height $g$, and set $R=Q/I$. Suppose that $R$ is a normal domain and $\id_R(I/I^2)<\infty$. Then one has
$$(g+1)[\omega_R]=0\quad\text{in }\operatorname{Cl}(R).$$
In particular, if $\operatorname{Cl}(R)$ is torsionfree, then $I$ is generated by a $Q$-regular sequence.}
\medskip

Theorem C connects the homological behavior of the conormal module with the canonical class of the quotient. When the divisor class group is torsionfree, the resulting relation forces the canonical class to vanish; hence $R$ is Gorenstein, and Briggs's theorem applies. This yields affirmative answers to our question for localizations at the homogeneous maximal ideals of several families of normal graded rings with torsionfree class groups, including generic and ladder determinantal rings, homogeneous coordinate rings of Schubert varieties in Grassmannians, and certain Pfaffian and toric rings; see Remark \ref{R3}.

In the graded setting, we obtain an affirmative answer to the question in the case of monomial ideals.

\medskip
\noindent\textbf{Theorem D.} \emph{ 
Let $Q=k[X_1,\ldots,X_n]$ be a polynomial ring over a field $k$,  let $I$ be a monomial ideal of $Q$, and $R=Q/I$. 
If $\id_R(I/I^2)<\infty$, then $R$ is a complete intersection.}
\medskip

The main idea of the proof is that finite injective dimension, together with Foxby equivalence, yields a tensor decomposition of the conormal module. By keeping track of multidegrees, this decomposition gives a factorization $I=JK$ for some monomial ideals $J$ and $K$.

We also study a related question for the first Koszul homology $\H_1(I)$, computed on a minimal generating set of $I$. For a Cohen--Macaulay almost complete intersection over a Gorenstein local ring, Kunz's results identify $\H_1(I)$ with the canonical module of $R$ \cite{Kunz}. We prove the converse: if $R$ is Cohen--Macaulay and $\H_1(I)\cong\omega_R$, then $I$ is an almost complete intersection. Under additional numerical hypotheses, we obtain the same conclusion from the weaker assumption $\id_R \H_1(I)<\infty$. These results connect the injective homological properties of the conormal module and Koszul homology with complete intersections and almost complete intersections, respectively.

The paper is organized as follows.
In Section 2, we consider the case of licci ideals and prove Theorem A.
In Section 3, we develop the numerical approach leading to Theorem B and study the first Koszul homology.
In Section 4, we treat the normal and graded cases and prove Theorems C and D.

\begin{conv}
    We use the following notation throughout the paper. Let $R$ be a Noetherian ring and let $M$ be a finitely generated $R$-module. We denote by $\mu_R(M)$ and $\ell_R(M)$ the minimal number of generators and the length of $M$, respectively. We write $\pd_R M$ and $\id_R M$ for the projective and injective dimensions of $M$. When $(R,\m,k)$ is a local ring, we set $\type R=\dim_k \Ext_R^{\depth R}(k,R)$ and write $e_R(M)$ for the multiplicity of $M$.
    For an ideal $I$, we denote by $K(I)$ the Koszul complex on a minimal generating set of $I$ and by $\H_i(I)$ its $i$-th homology module. We denote the height of $I$ by $\height(I)$. We omit subscripts and superscripts when the underlying ring is clear.
\end{conv}

\section{Licci ideals}

In this section, we consider the question for licci ideals, which form a natural generalization of complete intersection ideals. Ring-theoretic properties of the corresponding quotient rings are often inherited under linkage, and by combining these properties appropriately, we prove Theorem A. The key ingredient is the following lemma, which deduces the Gorenstein property from the associated primes of a tensor product involving canonical modules.

\begin{lem}\label{Ass tensor}
    Let $(R,\m ,k)$ be a generically Gorenstein Cohen--Macaulay local ring with a canonical module $\omega_R$.
If there exists a nonzero finitely generated $R$-module $M$ such that $\Ass_R(\omega_R\otimes_R\omega_R\otimes_R M)\subseteq\Ass(R)$, then $R$ is Gorenstein.
\end{lem}

\begin{proof}
Put $N=\omega_R\otimes_R\omega_R\otimes_R M$. Consider the map 
$$f: N\to N: x\otimes y\otimes m\mapsto y\otimes x\otimes m$$
that interchanges the two $\omega_R$-components.
Since $R$ is generically Gorenstein, $(\omega_R)_\p\cong R_\p$ for all $\p\in\Ass(R)$. In particular, $f_\p$ is equal to the identity map on $N_\p$. 
This means that $\Supp_R(\im(f-\id_N))\cap \Ass(R)=\emptyset$.
On the other hand, as $\im(f-\id_N)$ is a submodule of $N$, one has $\Ass_R(\im(f-\id_N))\subseteq\Ass_R(N)\subseteq\Ass(R)$ by assumption. By this, $\im(f-\id_N)=0$. Thus $f$ is the identity map, and so is the map
\begin{align}\label{switch}
(\omega_R/\m \omega_R)\otimes_k (\omega_R/\m \omega_R) \otimes_k (M/\m M)&\to (\omega_R/\m \omega_R)\otimes_k (\omega_R/\m \omega_R) \otimes_k (M/\m M) \\ 
\overline{x}\otimes \overline{y}\otimes \overline{m}&\mapsto \overline{y}\otimes \overline{x}\otimes \overline{m} \notag
\end{align}
obtained from $f$ by tensoring with $k$.
If $e_1,\ldots,e_r$ is a $k$-basis of $\omega_R/\m \omega_R$ and $d_1,\ldots,d_s$ is a $k$-basis of $M/\m M$, then $\{e_i\otimes e_j\otimes d_h\mid 1\le i,j\le r, 1\le h\le s\}$ is a $k$-basis of $(\omega_R/\m \omega_R)\otimes_k (\omega_R/\m \omega_R) \otimes_k (M/\m M)$.
As $M$ is nonzero, $s$ is positive.
Since \eqref{switch} is the identity map, $e_i\otimes e_j\otimes d_1=e_j\otimes e_i\otimes d_1$ for all $1\le i,j\le r$, which forces $r=1$.
Hence $\omega_R$ is cyclic, and therefore $R$ is Gorenstein.
\end{proof}

The assumption on the set of associated prime ideals in the lemma is satisfied if $M\otimes_R\omega_R\otimes_R\omega_R$ is maximal Cohen--Macaulay. This is analogous to the approaches to the Tachikawa conjecture taken in \cite{ABS, HH}. In fact, the following result is obtained as a direct consequence of the lemma. For example, taking $M=R=N$ in (1) recovers \cite[Theorem 2.1]{ABS} and \cite[Corollary 2.2]{HH}.

\begin{cor}\label{Tachikawa}
    Let $(R,\m ,k)$ be a generically Gorenstein Cohen--Macaulay local ring with canonical module $\omega_R$ and let $M,N$ be nonzero finitely generated $R$-modules. If one of the following holds, then $R$ is Gorenstein:
\begin{enumerate}[\rm (1)]
\item $\Ext_R^i(M\otimes_R\omega_R, \Hom_R(N,R))=0=\Ext_R^i(N,R)$ for all $1\le i\le \dim R$.
\item $\Ext_R^i(M\otimes_R\omega_R\otimes_R\omega_R, \Hom_R(N,\omega_R))=0=\Ext_R^i(N,\omega_R)$ for all $1\le i\le \dim R$.
\end{enumerate}
\end{cor}

\begin{proof}
In both cases (1) and (2), applying \cite[Lemma 3.4(1)]{LM} twice shows that $L=M\otimes_R N\otimes_R\omega_R\otimes_R\omega_R$ is maximal Cohen--Macaulay.
Indeed, for (1), $N\otimes_R\omega_R$ is maximal Cohen--Macaulay since $\Ext_R^i(N,R)=0$ for all $1\le i\le \dim R$. As $\Hom_R(N\otimes_R\omega_R,\omega_R)\cong \Hom_R(N,R)$, $L$ is maximal Cohen--Macaulay by $\Ext_R^i(M\otimes_R\omega_R, \Hom_R(N,R))=0$ for all $1\le i\le \dim R$.
One has $\Ass_R(L)\subseteq\Ass(R)$. Lemma \ref{Ass tensor} shows that $R$ is Gorenstein.
The proof of (2) is similar.
\end{proof}

Ideals $I$ and $J$ of a Noetherian ring are said to be \textit{linked} if there exists a regular sequence $\boldsymbol{x}$ contained in both $I$ and $J$ such that $I=(\boldsymbol{x}):J$ and $J=(\boldsymbol{x}):I$. An ideal is called \textit{licci} if it can be linked to a complete intersection ideal in finitely many steps.
The following theorem asserts that, over a Gorenstein local ring, a licci ideal whose conormal module has finite injective dimension must be a complete intersection.

\begin{thm}
    Let $Q$ be a Gorenstein local ring, let $I$ be a licci ideal of $Q$, and $R=Q/I$. If $\id_R(I/I^2)<\infty$, then $I$ is generated by a $Q$-regular sequence. In particular, if $Q$ is regular, then $R$ is a complete intersection.
\end{thm}

\begin{proof}
First, we assume that $R$ is Artinian and has an infinite residue field.
By assumption, $I$ is in the linkage class of a complete intersection ideal $J$ of $Q$ of grade $q=\dim Q$.
By \cite[Theorem 2.2]{HBI}, $I$ is perfect of grade $q$.
As $R$ is Artinian, $C:=I/I^2\cong \omega_R^{\oplus a}$ for some $a>0$, where $\omega_R$ is the canonical module of $R$.
It follows from \cite[Proposition 5.4]{HIW} that 
$$
a\ell(\omega_R\otimes_R \omega_R)-q\ell(R)=\ell(C\otimes_R \omega_R)-q\ell(R)=\ell(J/J^2\otimes_{Q/J} \omega_{Q/J})-q\ell(Q/J)=0
$$
since $J$ is a complete intersection ideal, and hence $\omega_{Q/J}\cong Q/J$ and $J/J^2\cong (Q/J)^{\oplus q}$.
We obtain 
\begin{equation}\label{licci eq1}
q\ell(R)=a\ell(\omega_R\otimes_R \omega_R)
\end{equation}
Since the residue field is infinite and $I$ is licci, \cite[Lemma 4.7]{JM} shows that there exist a Gorenstein local ring $B$, an ideal $K$ of $B$, and a regular sequence $\boldsymbol{x}=x_1, \ldots,x_s$ on both $B$ and $A=B/K$ such that $B/\boldsymbol{x}B\cong Q$, $A/\boldsymbol{x}A\cong R$, and $K$ is generically a complete intersection in $B$.
In particular, $(\boldsymbol{x}B+K)/\boldsymbol{x}B\cong I$ under $B/\boldsymbol{x}B\cong Q$.
Let $D=K/K^2$ and $\omega_A$ be a canonical module of $A$.
As $\boldsymbol{x}$ is a $B/K$-regular sequence, one has $\boldsymbol{x}B\cap K=\boldsymbol{x}K$, and hence 
$$
C=I/I^2\cong \frac{\boldsymbol{x}B+K}{\boldsymbol{x}B+K^2}\cong \frac{K}{K\cap(\boldsymbol{x}B+K^2)}=\frac{K}{(K\cap\boldsymbol{x}B)+K^2}=\frac{K}{\boldsymbol{x}K+K^2}\cong D/\boldsymbol{x}D.
$$
In what follows, we will use without further comment the invariants appearing in \cite[Section 4.7]{BH}.
The sequence $\boldsymbol{x}$ is a system of parameters of $D$.
It follows from \cite[Theorems 4.7.6 and 4.7.10]{BH} that 
\begin{equation}\label{licci eq2}
a\ell(R)=\ell(C)=\ell(D/\boldsymbol{x}D)\ge \chi(\boldsymbol{x}, D)=e(\boldsymbol{x}, D)
\end{equation}
Also, $\boldsymbol{x}$ is an $A$-regular sequence, and one has $\omega_R=\omega_A/\boldsymbol{x}\omega_A$.
Since $\height(K)=\dim B-\dim A=\dim Q-\dim R=q$ and $K$ is generically a complete intersection, we have $\rank_A(D)=q$ and $\rank_A(\omega_A\otimes_A \omega_A)=1$.
By \cite[Corollary 4.7.9 and Theorems 4.7.6 and 4.7.10]{BH}, the following equalities hold:
\begin{align}
&e(\boldsymbol{x}, D)=\rank_A(D)e(\boldsymbol{x}, A)=qe(\boldsymbol{x}, A)=q\chi(\boldsymbol{x}, A)=q\ell(A/\boldsymbol{x}A)=q\ell(R), \label{licci eq3} \\
&\ell(\omega_R\otimes_R \omega_R)\ge \chi(\boldsymbol{x}, \omega_A\otimes_A \omega_A)=e(\boldsymbol{x}, \omega_A\otimes_A \omega_A)=\rank_A(\omega_A\otimes_A \omega_A)e(\boldsymbol{x}, A)=e(\boldsymbol{x}, A)=\ell(R). \label{licci eq4}
\end{align}
Combining \eqref{licci eq1}, \eqref{licci eq2}, \eqref{licci eq3}, and \eqref{licci eq4}, we obtain
$$
q\ell(R)=a\ell(\omega_R\otimes_R \omega_R)\ge a\ell(R)\ge e(\boldsymbol{x}, D)=q\ell(R).
$$
Hence all inequalities are equalities. In particular, equality holds in \eqref{licci eq4}, and thus $\omega_A\otimes_A\omega_A$ is maximal Cohen--Macaulay by \cite[Theorem 4.7.10]{BH}.
Since $\Ass_A(\omega_A\otimes_A\omega_A)=\Ass(A)$, Lemma \ref{Ass tensor} states that $A$ is Gorenstein, and so is $R$. We see that $\pd_R(C)<\infty$, and $I$ is a complete intersection by \cite[Theorem 3.1]{B}. 

Next, we consider the case where $R$ is an arbitrary Artinian ring.
Let $\n$ be the maximal ideal of $Q$.
The ring $S=Q[X]_{\n Q[X]}$ is a Gorenstein local ring with infinite residue field, and $Q\to S$ is flat and local and its closed fiber $S/\n S$ is a field.
Then $R\to R\otimes_Q S=S/IS$ is a flat local ring homomorphism with closed fiber $S/\n S$.
Since $I$ is licci, so is $IS$ by \cite[Theorem 2.6(d)]{JM}.
One has $I/I^2\otimes_R (S/IS)\cong IS/(IS)^2$ and $I/I^2$ is a finite direct sum of copies of the canonical module of $R$. 
By \cite[Theorem 3.3.14]{BH}, $IS/(IS)^2$ is a finite direct sum of copies of the canonical module of $S/IS$, and hence it has finite injective dimension. 
By the previous paragraph, we have $\mu_Q(I)=\mu_S(IS)=\grade_S(IS)=\grade_Q(I)$.

Finally, we deal with the general case.
The condition $\id_R(I/I^2)<\infty$ shows that $R$ is Cohen--Macaulay.
Upon localization at a minimal prime of $I$, the Gorensteiness of $Q$, the licci property of $I$, and the finiteness of the injective dimension of $I/I^2$ are preserved; see \cite[Theorem 2.6(d)]{JM}. Hence, by the Artinian case, we see that $I$ is generically a complete intersection.
Now $I$ is in the linkage class of a complete intersection ideal $J$ of $Q$ and it is perfect by \cite[Theorem 2.2]{HBI}.
One has $I/I^2\cong\Hom_R(\omega_R,I/I^2)\otimes_R \omega_R$ as $I/I^2$ has finite injective dimension.
Since $J\otimes_Q \omega_{Q/J}\cong (J/J^2)\otimes_{Q/J} \omega_{Q/J}\cong (Q/J)^{\oplus \height(J)}$ is Cohen--Macaulay, so is 
$$
I\otimes_Q \omega_R\cong (I/I^2)\otimes_R \omega_R \cong \Hom_R(\omega_R,I/I^2)\otimes_R \omega_R\otimes_R \omega_R
$$ 
by \cite[Theorem 2.4]{HBI}.
By this, we have $\Ass_R(\Hom_R(\omega_R,I/I^2)\otimes_R \omega_R\otimes_R \omega_R)=\Ass(R)$. 
Lemma \ref{Ass tensor} shows that $R$ is Gorenstein.
By the same argument as at the end of the first paragraph, $I$ is a complete intersection.
\end{proof}

\section{Numerical approach}

In this section, we study numerical conditions under which finite injective dimension of the conormal module forces the complete intersection property. By relating the number of generators of the conormal module to the height, the type, and the first Koszul homology, we prove Theorem B by establishing Theorem \ref{main}, Corollary \ref{gGor}, and Proposition \ref{prime}. We also consider the analogous question for the first Koszul homology and obtain a converse to Kunz's result on almost complete intersections.
The first theorem of this section extends \cite[Theorem 4.8]{GT}, which gives an affirmative answer to the question for almost complete intersection ideals, one step further.

\begin{thm}\label{main}
    Let $Q$ be a regular local ring, let $I$ be a proper ideal of $Q$, and set $R=Q/I$. Assume $\mu(I)\le \height(I)+2$ and $\id_R(I/I^2)<\infty$. Then $R$ is a complete intersection.
\end{thm}

\begin{proof}
    Suppose, for a contradiction, that $R$ is not Gorenstein. By \cite[Theorem 4.8]{GT}, we may assume that $\mu:=\mu(I)=\height(I)+2$. The condition $\id_R(I/I^2)<\infty$ shows that $R$ is Cohen--Macaulay. For a minimal system of generators $\boldsymbol{x}=x_1, \ldots,x_\mu$ of $I$, let $\H_i(I)$ be the Koszul homology over $Q$ with respect to $\boldsymbol{x}$. We prove that $e(\H_1(I))=2e(R)$, where $e(-)$ denotes the multiplicity. It follows from  \cite[Proposition 2.1]{Kunz} that $\H_0(I)=R$, $\H_2(I)=\omega_R$, and $\H_i(I)=0$ for all $i>2$, where $\omega_R$ is the canonical module of $R$. For a minimal prime ideal $\p\in\Min_Q(R)$, $\boldsymbol{x}$ is not a system of parameters of $Q_\p$ because $\height(\p)=\height(I)<\mu$. By \cite[Theorem 4.7.6]{BH}, the Euler characteristic $\chi(\boldsymbol{x}, Q_\p)$ is zero. Noting that $\H_i(I)_\p=0$ for all $i>2$, we obtain $\ell(\H_1(I)_\p)=\ell(\H_0(I)_\p)+\ell(\H_2(I)_\p)=2 \ell(R_\p)$. It follows from \cite[Corollary 4.7.8]{BH} that $e(\H_1(I))=\sum_{\p\in\Min_Q(R)} \ell(\H_1(I)_{\p}) e(R/\p)=2\sum_{\p\in\Min_Q(R)} \ell(R_{\p}) e(R/\p)=2e(R)$.

    There exists an exact sequence $\H_1(I)\to R^{\oplus \mu} \to I/I^2\to 0$ and a minimal covering $\omega_R^{\oplus a} \twoheadrightarrow I/I^2$ for some $a$. One has the inequality 
$$
\mu e(R)=e(R^{\oplus \mu})\le e(\H_1(I))+e(I/I^2)\le e(\H_1(I))+e(\omega_R^{\oplus a})=2e(R)+ae(R)=(a+2)e(R), 
$$
    and hence $\mu\le a+2$. On the other hand, an analogous argument to the proof of \cite[Theorem 4.8]{GT} shows that $\mu=a\type(R)$. Then we get $a\type(R)\le a+2$. Since $R$ is not Gorenstein, $\type(R)\ge 2$ holds. The only possible cases in which the inequality $a\type(R)\le a+2$ can hold are $(\type(R), a)= (2, 1), (2, 2), (3, 1)$. In the first case, we have $\height(I)=\mu-2=a\type(R)-2=0$ and hence $I=0$, which is a contradiction. In the second case, $\mu=4$ and $\height(I)=2$ hold. By the Hilbert--Burch theorem, $R$ has a minimal free resolution as follows: $$0\to Q^{\oplus 3}\to Q^{\oplus 4}\to Q\to R \to 0.$$
    This implies that $\type(R)=3$, which is a contradiction. In the last case, we have $\mu=3$ and $\height(I)=1$. Since $I$ is a perfect ideal, such a situation cannot occur.
    Therefore $R$ is Gorenstein. We see that $\pd_R(I/I^2)<\infty$, and $R$ is a complete intersection by \cite[Theorem 3.1]{B}.
\end{proof}

\begin{rem}\label{R1}
    Let $Q$ be a regular local ring, let $I$ be a proper ideal of $Q$, and set $R=Q/I$. Suppose $\id_R(I/I^2)<\infty$. Then there is a surjection $\omega_R^{\oplus a} \twoheadrightarrow I/I^2$. Assume further that $R$ is generically Gorenstein. If $a=1$, then $I$ is generated by a regular element. Indeed, $I_\p$ is a principal ideal for all minimal primes $\p$ of $I$. Therefore $\height(I)=1$ and since $R$ is Cohen--Macaulay, $I$ itself is a principal ideal. 
\end{rem}

By this remark, we obtain the following corollary and proposition.

\begin{cor}\label{gGor}
    Let $Q$ be a regular local ring, let $I$ be a proper ideal of $Q$, and set $R=Q/I$. Assume $R$ is generically Gorenstein, and $\mu(I)\le \height(I) + \type(R)$. If $\id_R(I/I^2)< \infty$, then $R$ is a complete intersection. 
\end{cor}   
\begin{proof} 
   Note that by the assumptions, $(I/I^2)_\p$ is a free $R_\p$-module. Therefore, $I$ is generically a complete intersection.
   By localizing $\H_1(I)\to R^{\oplus \mu} \to I/I^2\to 0$ in a minimal prime ideal of $I$ and counting ranks, \cite[2.5]{H2} yields the inequality $\mu\le a+\type(R)$, where $a$ is the same as in the proof of Theorem \ref{main}. Since $\mu=a \type(R)$, the only possibilities  for the inequality are $a=1$, $a=2=\type(R)$, or $\type(R)=1$. For $\type(R)\le 2$, the assertion follows from Theorem \ref{main}, and for $a=1$, it  follows from Remark \ref{R1}.
\end{proof}

\begin{prop}\label{prime}
   Let $Q$ be a regular local ring, let $I$ be a proper ideal of $Q$, and set $R=Q/I$. Suppose that $\mu(I)$ is a prime number. Assume either $Q$ is unramified or $R$ is generically Gorenstein. If $\id_R(I/I^2)<\infty$, then $I$ is a complete intersection.
\end{prop}

\begin{proof}
There is a surjection $\omega_R^{\oplus a} \twoheadrightarrow I/I^2$ with $\mu(I)=a \type(R)$. Since $\mu(I)$ is prime, either $\type(R)=1$ or $a=1$. If $\type(R)=1$, then $R$ is Gorenstein, and hence $R$ is a complete intersection. Suppose that $a=1$. If $R$ is generically Gorenstein, the assertion follows from Remark \ref{R1}. Suppose $Q$ is unramified. Applying $\Hom_R(\omega_R,-)$ to $\omega_R \twoheadrightarrow I/I^2$ gives a surjection $R\twoheadrightarrow \Hom_R(\omega_R, I/I^2)$. Hence $R/J\cong\Hom_R(\omega_R, I/I^2)$ for some ideal $J$, and 
$$
\omega_R/J \omega_R\cong \omega_R \otimes_R (R/J)\cong \omega_R \otimes_R\Hom_R(\omega_R, I/I^2)\cong I/I^2
$$
as $\id_R(I/I^2)<\infty$.
If $J\ne 0$, then we can choose $\p\in\Ass(J)\subseteq\Ass(R)$. Then $(R/J)_\p$ is a free $R_\p$-module as $\pd_{R_\p} (R/J)_\p=\pd_{R_\p} \Hom_R(\omega_R, I/I^2)_\p\le \depth R_\p=0$, which contradicts $J_\p\ne 0$. We obtain $I/I^2\cong \omega_R$.

The natural exact sequence $0\to I/I^2\to Q/I^2\to R\to 0$ is a commutative extension in the sense of \cite{Fo}, and \cite[Theorem]{Fo} shows that $Q/I^2$ is Gorenstein. It follows from \cite[Theorem 2]{H3} that $\height(I)\le 1$ since $Q$ is unramified. As $I$ is unmixed, $\height(I)\le 1$ implies that $\mu(I)\le 1$, a contradiction.
\end{proof}

\begin{rem}
     Let $(Q,\n,k)$ be a regular local ring. Let $A=(f_{ij})$ be an $m\times n$ matrix with $f_{ij}\in \n$, $1\le i \le m$, $1 \le j \le n$, and $ m\le n$. Let $I$ be an ideal generated by $m$-minors of $A$. Assume $R=Q/I$ is generically Gorenstein. If $\id_R(I/I^2)<\infty$ and $\height(I)=n-m+1$, then $I$ is a complete intersection. Indeed, by Eagon--Northcott, one has $\mu(I)=\binom{n}{m}$ and  $\type(R)=\binom{n-1}{m-1}$; see \cite[Appendix 2]{E}. Then the equality $\mu(I)=\mu(\Hom_R(\omega_R,I/I^2)) \type(R)$ shows that $\mu(\Hom_R(\omega_R,I/I^2))=\frac{n}{m}$. Since $I$ is generically a complete intersection,  $\rank(I/I^2) = \height(I)=n-m+1$. It follows that $\frac{n}{m} \ge n-m+1$. This inequality holds only for $m=1$ or $m=n$. In both cases, $I$ is a complete intersection.
\end{rem}

Gulliksen proved that if $\pd_Q(I)<\infty$ and $\H_1(I)$ is free over $R$, then $I$ is generated by a regular sequence; see \cite[Proposition 1.4.9]{GL}. Briggs subsequently strengthened this result by showing that the same conclusion follows from the weaker assumption $\pd_R \H_1(I)<\infty$ \cite[Theorem B]{B}. Thus, under the finite projective dimension hypothesis on $I$, both the conormal module and the first Koszul homology detect the complete intersection property through finiteness of their projective dimensions over $R$. By \cite[Proposition 2.1]{Kunz}, $\H_1(I)\cong \omega_R$ if $R$ is Cohen--Macaulay and $\mu(I)=\height(I) +1$. Therefore, finiteness of injective dimension of $\H_1(I)$ does not necessarily force the complete intersection property. Recall that an ideal $I$ is called an almost complete intersection if $\mu(I)\le \height(I)+1$. 

\begin{ques}\label{Q2}
    Let $Q$ be a Gorenstein local ring, $I$ a proper ideal of $Q$ of finite projective dimension, and set $R=Q/I$. If $\H_1(I)$ is nonzero and $\id_R\H_1(I)<\infty$, is $I$ an almost complete intersection?
\end{ques}

We answer Question \ref{Q2} in the following cases.

\begin{thm}\label{kunz}
    Let $Q$ be a Gorenstein local ring and $I$ a proper ideal of $Q$. Let $R=Q/I$ be Cohen--Macaulay, and let $\omega_R$ be a canonical module of $R$. Assume one of the following conditions holds.
\begin{enumerate}
    \item There is an isomorphism $\H_1(I)\cong \omega_R$.
    \item One has $\height(I)=2$, $\pd_QI<\infty$, and $\id_R\H_1(I)<\infty$.
    \item One has $\mu(I) \le \height(I) + 2$, $\pd_QI<\infty$, and $\id_R(\H_1(I))<\infty$.
\end{enumerate} Then $I$ is an almost complete intersection. 
\end{thm}
\begin{proof}
    (1) Let $x\in Q$ be regular on both $Q$ and $R$. The exact sequence of Koszul complexes $0\to K(I) \overset{x}\to K(I) \to K(I)/xK(I) \to 0$ induces a long exact sequence $$\dots \to \H_i(I) \overset{x}\to \H_i(I) \to \H_i(K(I)/xK(I))\to \H_{i-1}(I)\overset{x}\to \H_{i-1}(I) \to \dots.$$
    Since $\H_1(I)\cong \omega_R$ and $x$ is regular on both $R$ and $\omega_R$, the zeroth terms and the first terms of the long exact sequence above break into short exact sequences $0\to \H_1(I) \overset{x}\to \H_1(I) \to \H_1(K(I)/xK(I))\to 0$ and $0\to R \overset{x}\to R \to Q/(x,I) \to 0$. Thus, we get $\H_1(K(I,x))\cong \H_1(K(I)/xK(I)) \cong \omega_R/x\omega_R$. This shows that by modding out by a regular sequence, we may assume that $R$ is Artinian. Since $\H_1(I)\neq 0$, $R$ is not a complete intersection. Therefore, we may assume $\mu(I)=d+r$ where $d=\dim Q$ and $r\ge 1$. We show that $r=1$. It is enough to prove $\H_i(I)=0$ for $i\ge 2$. Let $\chi(I,R)=\sum_{i=0}(-1)^i\ell(\H_i(I))$ be the Euler characteristic of $R$. By Serre's theorem \cite[Theorem 4.7.6]{BH}, $\chi(I,R)=0$. Since by assumption $\ell(\H_0(I))= \ell(\H_1(I))$, we have $\chi_2(I,R)=\sum_{i=2}(-1)^i\ell(\H_i(I))=0$. It follows from \cite[Chapter IV, Appendix II]{S} that $\H_i(I)=0$ for all $i\ge 2$. This finishes the proof.

    (2)  If $\H_1(I)=0$, there is nothing to prove, so we assume $\H_1(I)\neq 0$. By \cite[Corollary 2.1]{Hu} and the assumptions, we have $\H_1(I)\cong \omega_R^{\oplus a}$. One observes that $a\type(R)=\mu(\H_1(I))\le \beta_2^Q(R)=\type(R)$. Therefore, we get $a\le 1$. Now, the result follows from (1).

    (3) Let $c=\mu(I)-\height(I)$. By \cite[Proposition 2.1]{Kunz} one has $\H_c(I)\cong \omega_R$. By \cite[Theorem A]{MRS}, there is a duality $$\Hom_R(\Hom_R(\H_i(I),\omega_R),\omega_R)\cong \Hom_R(\H_{c-i}(I),\omega_R),$$ for all $i\ge 0$. Since $c\le 2$ by assumption, it is enough to show that the case $c=2$ is not an option. If $c=2$, by \cite[Remark 1.3.2]{V3}, $\H_1(I)$ is Cohen--Macaulay. Therefore, by using the duality we have $\H_1(I)\cong \Hom_R(\H_1(I),\omega_R)$. This shows that $\H_1(I)$ is a free module and so that $R$ is a complete intersection by \cite[Theorem B]{B}, a contradiction. 
\end{proof}

\begin{rem}
    Keeping assumptions of Theorem \ref{kunz}, if $\mu(I)=\height(I)+2$ and $\id_R\H_1(I)<\infty$, then $Q \twoheadrightarrow R$ is a quasi-complete intersection homomorphism in the sense of \cite{AHS}. Indeed, as in the proof of Theorem \ref{kunz}(3), $\H_i(I)$ is free for all $i$. Since $\H_2(I)\cong \omega_R$, we see that $R$ is Gorenstein and $\H_2(I)\cong R$. Let $\H_1(I)\cong R^{\oplus r}$. Let $\boldsymbol{x}$ be a minimal system of generators of $I$. Localizing at a minimal prime $\p$ of $I$, the vanishing of $\chi(\boldsymbol{x},Q_\p)$ gives $0=\ell(\H_0(I)_\p)-\ell(\H_1(I)_\p)+\ell(\H_2(I)_\p)=(2-r)\ell(R_\p)$, and hence $r=2$. By the Koszul duality for Gorenstein rings, one has an isomorphism $\H_1(I) \overset{\cong}\longrightarrow \Hom_R(\H_1(I),\H_2(I))$ given by $[z] \longmapsto ([w]\mapsto [z\wedge w])$. This shows that $\bigwedge^2 \H_1(I) \cong \H_2(I)$ and therefore, $\bigwedge^* \H_1(I) \cong \H_*(I)$ as desired.
\end{rem}

\section{Normal and graded cases}

In this section, we consider the question in two further settings: normal domain quotients and monomial ideals in polynomial rings. For normal domains, the divisor class group provides a natural way to measure the obstruction coming from the canonical module, and this leads to Theorem C. We then turn to the graded case and prove Theorem D for monomial ideals. The key ingredient for the normal case is the following lemma, which allows us to control bases of modules of finite projective dimension in codimension one.

\begin{lem}\label{basislemma}
Let $R$ be a local domain, let $K$ be the fraction field of $R$, and let $M$ be a finitely generated $R$-module with finite projective dimension.
There exists a $K$-basis $e_1,\ldots,e_r$ of $M\otimes_R K$ satisfying the following condition: for every prime ideal $\p$ of $R$ such that $M_\p$ is free, and for every $R_\p$-basis $f_1,\ldots,f_r$ of $M_\p$, the determinant of the change of basis matrix from $e_1,\ldots,e_r$ to $f_1,\ldots,f_r$, viewed over $K$, is a unit in $R_\p$.
\end{lem}

\begin{proof}
We may assume $r>0$.
We prove this by induction on the projective dimension of $M$. If $M$ is projective, then, since $R$ is local, $M$ is free, so we may simply choose a basis $e_1,\ldots,e_r$ of $M$ over $R$.
Suppose $\pd_R M>0$ and take an exact sequence $0\to N\to R^{\oplus s} \to M\to 0$.
Put $(-)_K:=(-)\otimes_R K$.
Let $c_1,\ldots,c_s$ be the standard basis of $R^{\oplus s}$.
By the induction hypothesis, we may choose a $K$-basis $d_1,\ldots,d_{s-r}$ of $N_K$ satisfying the condition in the statement.
We choose a $K$-basis $e_1,\ldots,e_r$ of $M_K$ and lifts $\widetilde{e}_1,\ldots,\widetilde{e}_r\in R_K^{\oplus s}$ of $e_1,\ldots,e_r$ such that the determinant of the change of basis matrix from $\widetilde{e}_1,\ldots,\widetilde{e}_r, d_1,\ldots,d_{s-r}$ to $c_1,\ldots,c_s$ is $1$; that is, $\det A=1$ if $(\widetilde{e}_1,\ldots,\widetilde{e}_r \ d_1 \cdots  d_{s-r})^t=A(c_1 \cdots c_s)^t$.

Let $\p\in\spec R$ be such that $M_\p$ is free. Then so is $N_\p$.
Let $f_1,\ldots,f_r$ be an $R_\p$-basis of $M_\p$ and $g_1,\ldots,g_{s-r}$ be an $R_\p$-basis of $N_\p$.
Choose lifts $\widetilde{f}_1,\ldots,\widetilde{f}_r\in R_\p^{\oplus s}$ of $f_1,\ldots,f_r$. Then $\widetilde{f}_1,\ldots,\widetilde{f}_r,g_1,\ldots,g_{s-r}$ is an $R_\p$-basis of $R_\p^{\oplus s}$.
When we write $(f_1 \cdots f_r)^t=B(e_1 \cdots e_r)^t$ and $(g_1 \cdots g_{s-r})^t=C(d_1 \cdots d_{s-r})^t$ in $K$, the differences $\widetilde{f}_i-\sum_jB_{ij}\widetilde{e}_j$ belong to $N_K$. Hence, for some matrix $D$ over $K$,
$$
\begin{pmatrix} \widetilde{f}_1\\ \vdots \\ \widetilde{f}_r \\ g_1 \\ \vdots \\ g_{s-r}\end{pmatrix}
=\begin{pmatrix} B & D \\ 0 & C\end{pmatrix}
\begin{pmatrix} \widetilde{e}_1\\ \vdots \\ \widetilde{e}_r \\ d_1 \\ \vdots \\ d_{s-r}\end{pmatrix}
=\begin{pmatrix} B & D \\ 0 & C\end{pmatrix}A\begin{pmatrix} c_1\\ \vdots \\c_s \end{pmatrix}.
$$
Since both $\widetilde{f}_1,\ldots,\widetilde{f}_r,g_1,\ldots,g_{s-r}$ and $c_1,\ldots,c_s$ are $R_\p$-bases of $R_\p^{\oplus s}$, the determinant of the corresponding change of basis matrix,
$$
\det\left(\begin{pmatrix} B & D \\ 0 & C\end{pmatrix}A\right)
=\det A\cdot\det B\cdot\det C,
$$
is a unit in $R_\p$. Since $\det A=1$ and $\det C$ is a unit in $R_\p$, it follows that $\det B$ is also a unit in $R_\p$.
\end{proof}

We briefly recall the basic facts and conventions concerning the divisor class group that will be used below; see, for example, \cite{Fo2} for further details.
Let $R$ be a Noetherian normal domain with fraction field $K$. For each prime ideal $\p$ of $R$ with $\height\p=1$, the localization $R_\p$ is a discrete valuation ring, and we denote its normalized valuation by $v_\p$. The divisor class group $\operatorname{Cl}(R)$ is the quotient of the free abelian group $\operatorname{Div}(R)=\bigoplus_{\height\p=1}\Z[\p]$ by the subgroup of principal divisors $\operatorname{div}(a)=\sum_{\height\p=1}v_\p(a)[\p]$, where $0\ne a\in K$.
Let $M$ be a finitely generated reflexive $R$-module of rank one. Via the natural map $M\to M\otimes_RK\cong K$, we may regard $M$ as an $R$-submodule of $K$. As $M$ is reflexive, it can be identified with a divisorial ideal, and hence we may associate to it a divisor class $[M]\in\operatorname{Cl}(R)$; see, for example, \cite[Page 26]{Fo2} and \cite[Page 205]{Lip}.
More explicitly, fix a nonzero element $\xi\in M\otimes_RK$, and, for each prime ideal $\p$ of height one, choose $a_\p\in K$ such that $M_\p=R_\p(a_\p\xi)$. Then $[M]$ is defined to be the class of the divisor $\sum_{\height\p=1}v_\p(a_\p)[\p]$. This class is independent of the choice of $\xi$. Moreover, $[M]=0$ in $\operatorname{Cl}(R)$ if and only if $M$ is isomorphic to $R$.

\begin{thm}\label{classgroup}
Let $Q$ be a regular local ring, let $I$ be an ideal of $Q$ of height $g$, and set $R=Q/I$.
If $R$ is a normal domain and $\id_R(I/I^2)<\infty$, then $(g+1)[\omega_R]=0$ in $\operatorname{Cl}(R)$.
Moreover, if $\operatorname{Cl}(R)$ is torsionfree, $R$ is a complete intersection.
\end{thm}

\begin{proof}
We may assume $I\ne 0$.
The condition $\id_R(I/I^2)<\infty$ shows that $R$ is Cohen--Macaulay.
Let $F$ be a minimal free resolution of $R$ over $Q$, and then $\omega_R:=\H^g(\Hom_Q(F,Q))$ is a canonical module of $R$.
As $\id_R(I/I^2)<\infty$, there is the natural isomorphism $\Phi:\omega_R\otimes_R M \to I/I^2$, where $M=\Hom_R(\omega_R,I/I^2)$ has finite projective dimension and rank $g$.
Since $R$ is an integral domain, $I$ is prime and $K:=Q_I/I Q_I$ is the fraction field of $R$.
Put $(-)_K:=(-)\otimes_R K$.
Let $0\ne \xi\in(\omega_R)_K$ and let $e_1,\ldots,e_g$ be a $K$-basis of $M_K$ as in Lemma \ref{basislemma}.
Applying $(-)_K$ to $\Phi$ induces an isomorphism
$$
\Phi_K: (\omega_R)_K\otimes_K M_K \xrightarrow{\simeq} (\omega_R\otimes_R M)_K \xrightarrow{\simeq} (I/I^2)_K
$$
Let $\boldsymbol{x}=x_1,\ldots, x_g\in IQ_I$ be inverse images of 
$$
\Phi_K(\xi\otimes e_1),\ldots, \Phi_K(\xi\otimes e_g)\in (I/I^2)_K\cong IQ_I/I^2Q_I.
$$
Since $\xi\otimes e_1,\ldots,\xi\otimes e_g$ is a $K$-basis of $(\omega_R)_K\otimes_K M_K$, $\boldsymbol{x}$ generates $IQ_I$.
For the Koszul complex $X:=K(\boldsymbol{x},Q_I)$ over $Q_I$ with respect to $\boldsymbol{x}$, let $\theta_{\boldsymbol{x}}\in \Hom_{Q_I}(X_g, Q_I)$ denote the canonical homomorphism that sends the distinguished exterior product generator to $1$.
The image of $\theta_{\boldsymbol{x}}$ is nonzero in $\H^g(\Hom_{Q_I}(X, Q_I))$. 
Localizing $F$ at $I$, $F_I$ is a free resolution of $K$ over $Q_I$. Choose a comparison map $u:F_I\to X$ lifting $\id_K$. 
Dualizing and taking cohomology gives an isomorphism 
$$
\alpha:=\H^g(\Hom_{Q_I}(u, Q_I)):\H^g(\Hom_{Q_I}(X, Q_I))\xrightarrow{\simeq}
\H^g(\Hom_{Q_I}(F_I,Q_I))=(\omega_R)_K.
$$
Write $\alpha([\theta_{\boldsymbol{x}}])=\gamma\xi$ for some $0\neq\gamma\in K$.

Let $\p\in\spec R$ with $\height(\p)=1$ and let $\q\in\spec Q$ with $\p=\q/I$.
As $R$ is normal, $R_\p=Q_\q/IQ_\q$ is regular.
Also, so is $Q_\q$, which shows that $IQ_\q$ is a complete intersection ideal of grade $g$.
There is $\boldsymbol{y}=y_1,\ldots,y_g\in I$ such that $(y_1,\ldots,y_g)Q_\q=IQ_\q$.
Viewing $\boldsymbol{x}$ and $\boldsymbol{y}$ as column vectors in $Q_I^{\oplus g}$, there is the transition matrix $A$ between two minimal systems of generators of $IQ_I$ with entries in $Q_I$ such that 
\begin{equation}\label{basechange}
\begin{pmatrix} y_1\\ \vdots \\ y_g \end{pmatrix}
=A\begin{pmatrix} x_1\\ \vdots \\ x_g \end{pmatrix}.
\end{equation}
Similarly, for the Koszul complexes $Y:=K(\boldsymbol{y},Q_I)$ and  $Z:=K(\boldsymbol{y},Q_\q)$ over $Q_I$ and $Q_\q$, respectively, let $\theta_{\boldsymbol{y}}\in \Hom_{Q_I}(Y_g, Q_I)$ and $\zeta\in \Hom_{Q_\q}(Z_g, Q_\q)$ denote the canonical homomorphisms that send the distinguished exterior product generator to $1$.
By the definition of $A$, there is a commutative diagram
\[
  \xymatrix@C=35pt@R=20pt
  {
    0 \ar[r]
    & Y_g \ar[r] \ar[d]_{\times \det A}^{\cong} 
    & Y_{g-1} \ar[r] \ar[d]^{\cong} 
    & \cdots \ar[r]
    & Y_1 \ar[r]^{(y_1 \cdots y_g)}  \ar[d]_{A^t}^{\cong} 
    & Y_0 \ar@{=}[d] \ar[r]
    & 0\\
    0 \ar[r]
    & X_g \ar[r]
    & X_{g-1} \ar[r]
    & \cdots \ar[r]
    & X_1 \ar[r]^{(x_1 \cdots x_g)}
    & X_0 \ar[r]
    & 0.\\
  }
\]
Let $\psi:Y\to X$ denote the above morphism of Koszul complexes. Dualizing and taking cohomology gives an isomorphism
$$
\beta:=\H^g(\Hom_{Q_I}(\psi,Q_I)):\H^g(\Hom_{Q_I}(X, Q_I))\xrightarrow{\simeq}\H^g(\Hom_{Q_I}(Y,Q_I)).
$$
Since $\psi_g$ is multiplication by $\det A$, we have $\beta([\theta_{\boldsymbol{x}}])=(\det A)[\theta_{\boldsymbol{y}}]$.
Choose a comparison map $v:F_\q\to Z$ lifting $\id_{R_\p}$, and let
$$
\eta:=\H^g(\Hom_{Q_\q}(v,Q_\q))([\zeta])\in \H^g(\Hom_{Q_\q}(F_\q,Q_\q))=(\omega_R)_\p.
$$
Since $\H^g(\Hom_{Q_\q}(v,Q_\q))$ is an isomorphism and $[\zeta]$ is an $R_\p$-basis of $\H^g(\Hom_{Q_\q}(Z,Q_\q))$, $\eta$ is an $R_\p$-basis of $(\omega_R)_\p$.
After localizing $v$ at $I$, we obtain a comparison map $v_I:F_I\longrightarrow Y$.
Both $u$ and $\psi\circ v_I:F_I\to X$ lift $\id_K$, and hence they are chain homotopic. 
Therefore 
$$
\alpha=\H^g(\Hom_{Q_I}(v_I,Q_I))\circ\beta.
$$
Since $\beta([\theta_{\boldsymbol{x}}])=(\det A)[\theta_{\boldsymbol{y}}]$, it follows that the image of $\eta$ in $(\omega_R)_K$ is
$$
\H^g(\Hom_{Q_I}(v_I,Q_I))([\theta_{\boldsymbol{y}}])
=
(\det A)^{-1}\alpha([\theta_{\boldsymbol{x}}])
=
(\det A)^{-1}\gamma\xi.
$$
For the reader's convenience, we summarize the above discussion in the following commutative diagram:
\[
  \xymatrix@C=35pt@R=30pt
  {
  (\omega_R)_\p=\H^g(\Hom_{Q_\q}(F_\q,Q_\q)) \ar[r]^-{\text{localize}}
  & \H^g(\Hom_{Q_I}(F_I,Q_I))=(\omega_R)_K
  &
  \\
  \H^g(\Hom_{Q_\q}(Z,Q_\q)) \ar[u]^{\H^g(\Hom(v,Q_\q))}_{\cong}\ar[r]^-{\text{localize}}
  & \H^g(\Hom_{Q_I}(Y,Q_I)) \ar[u]^{\H^g(\Hom(v_I,Q_I))}_{\cong}
  & \H^g(\Hom_{Q_I}(X,Q_I)). \ar[l]^{\beta}_{\quad\cong} \ar[ul]_{\alpha}^{\cong\quad} 
  }
\]
In this diagram, the elements we consider correspond as follows:
\[
  \xymatrix@C=90pt@R=20pt
  {
  \eta \ar@{|->}[d] \ar@{|-->}[r]
  & (\operatorname{det}A)^{-1}\gamma \xi
  & 
  \\
  [\zeta] \ar@{|->}[r]
  & [\theta_{\boldsymbol{y}}] \ar@{|->}[r]
  & (\operatorname{det}A)^{-1}[\theta_{\boldsymbol{x}}] \ar@{|->}[ul]
  }
\]
By commutativity, the dashed correspondence is given by localization, which means that 
\begin{equation}\label{localize}
(\omega_R)_\p=R_\p\eta=R_\p((\det A)^{-1}\gamma)\xi\subseteq (\omega_R)_K.
\end{equation}

For the isomorphism
$$
\Phi_\p: (\omega_R)_\p\otimes_{R_\p} M_\p \xrightarrow{\simeq} (\omega_R\otimes_R M)_\p \xrightarrow{\simeq} (I/I^2)_\p
$$
induced by localizing $\Phi$ at $\p$, there exists an $R_\p$-basis $f_1,\ldots,f_g$ of $M_\p$ such that $\Phi_\p(\eta\otimes f_i)=\overline{y_i}$ in $(I/I^2)_\p$.
As $e_1,\ldots,e_g$ were chosen as in Lemma \ref{basislemma}, the determinant of the matrix $B$ appearing in the change of basis 
$$
\begin{pmatrix} f_1\\ \vdots \\ f_g \end{pmatrix}
=B\begin{pmatrix} e_1\\ \vdots \\ e_g \end{pmatrix}
$$
over $K$ is a unit in $R_\p$. 
Noting that $\eta=(\det A)^{-1}\gamma\xi$ in $(\omega_R)_K$, one has
$$
\begin{pmatrix} \eta\otimes f_1\\ \vdots \\ \eta\otimes  f_g \end{pmatrix}
=(\det A)^{-1}\gamma\begin{pmatrix} \xi\otimes f_1\\ \vdots \\ \xi\otimes  f_g \end{pmatrix}
=(\det A)^{-1}\gamma B\begin{pmatrix}  \xi\otimes e_1\\ \vdots \\  \xi\otimes e_g \end{pmatrix}
$$
as a vector in $((\omega_R)_K\otimes_K M_K)^{\oplus g}$.
Applying $\Phi_K$ to both sides yields the change of basis
$$
\begin{pmatrix} \overline{y_1}\\ \vdots \\ \overline{y_g} \end{pmatrix}
=(\det A)^{-1}\gamma B\begin{pmatrix} \overline{x_1}\\ \vdots \\ \overline{x_g} \end{pmatrix}
$$
in $(I/I^2)_K$.
Combining this with \eqref{basechange}, $A$ and $(\det A)^{-1}\gamma B$ are equal as matrices over $K$, and taking determinants yields $\det A=(\det A)^{-g}\gamma^g \det B$.
Since $\det B$ is a unit of $R_\p$, we have $(\det A)^{g+1}R_\p=\gamma^g R_\p$ as an $R_\p$-submodule of $K$.

By \eqref{localize}, the coefficient of $[\omega_R]$ at $\p$ with respect to the fixed basis $\xi$ is $v_\p((\det A)^{-1}\gamma)=v_\p(\gamma)-v_\p(\det A)$.
Since $(\det A)^{g+1}R_\p=\gamma^gR_\p$, we have $(g+1)v_\p(\det A)=g\,v_\p(\gamma)$,
and hence
$$
(g+1)\bigl(v_\p(\gamma)-v_\p(\det A)\bigr)=(g+1)v_\p(\gamma)-(g+1)v_\p(\det A)=(g+1)v_\p(\gamma)-g v_\p(\gamma)=v_\p(\gamma).
$$
Thus $(g+1)$ times a divisor representing $[\omega_R]$ is the principal divisor
$\operatorname{div}(\gamma)$, and hence it is zero in $\operatorname{Cl}(R)$.

Finally, if $\operatorname{Cl}(R)$ is torsionfree, then $[\omega_R]=0$ in $\operatorname{Cl}(R)$.
Hence $R$ is Gorenstein, and thus $\pd_R(I/I^2)<\infty$.
Then $R$ is a complete intersection by \cite[Theorem 3.1]{B}. 
\end{proof}

A related connection between the conormal module and the canonical module was established by Hong, Noh, and Vasconcelos \cite[Proposition 4.4]{HNV} in the affine characteristic-zero setting.

\begin{rem}\label{R3}
The following (standard graded) Cohen--Macaulay normal domains over a field are known to have torsionfree class groups:
\begin{itemize}
\item generic/ladder determinantal rings; see \cite[Corollary 3.5]{SWSS};
\item homogeneous coordinate rings of Schubert varieties in Grassmannians; see \cite[Theorem 8.1]{BV};
\item Pfaffian rings (and, more generally, certain cogenerated Pfaffian rings); see \cite[Theorem 2.5]{Ne};
\item various normal toric rings, including normal edge rings and toric rings 
associated to compressed lattice polytopes, such as Hibi rings and stable set 
rings of perfect graphs; see \cite[Theorem 1.1]{HiMa} and \cite[Corollary 3.4]{Ma}.
\end{itemize}
Thus, by Theorem \ref{classgroup}, for the localizations (and completions) of these rings at their homogeneous maximal ideals, the defining ideal is a complete intersection if and only if the conormal module has finite injective dimension.
Indeed, if the conormal module of the above ring $R$ has finite injective dimension, then by applying the theorem to the localization at the homogeneous maximal ideal $\m$, we obtain $(g+1)[\omega_{R_\m}]=0$ in $\operatorname{Cl}(R_\m)$. From this, one can deduce that $(g+1)[\omega_R]=0$ in $\operatorname{Cl}(R)$.
\end{rem}

Finally, we conclude the paper by considering the case of monomial quotients. The key idea is that, if the conormal module has finite injective dimension, then the Foxby equivalence implies that the monomial ideal can be factored as a product of two monomial ideals. This theorem refines \cite[Corollary 5.7]{A}.

\begin{thm}\label{monomial over k}
Let $Q=k[X_1,\ldots,X_n]$ be a polynomial ring over a field $k$,  let $I$ be a monomial ideal of $Q$, and $R=Q/I$. 
If $\id_R(I/I^2)<\infty$, then $R$ is a complete intersection.
\end{thm}

\begin{proof}
If $I=0$, there is nothing to prove. Hence we may assume that $I\ne0$.
Let $\n=(X_1,\ldots,X_n)$ be the maximal ideal of $Q$. 
As in, for example, \cite[Proposition 1.5.15]{BH}, the relevant information about the graded ring $Q$ and graded $Q$-modules can be recovered from their localizations at $\n$. We will use this fact freely throughout the proof; in particular, whenever convenient, we work over $Q_\n$ and then translate the resulting information back to $Q$ and its graded modules.
The proof is divided into two steps, the Artinian case and the general case. We first record a fact that will be used in both steps.

Let $\boldsymbol{e}_1,\ldots, \boldsymbol{e}_n\in\Z^{\oplus n}$ be the standard basis.
We set $\deg X_i=\boldsymbol{e}_i$ and regard $Q$ and $R$ as $\mathbb{N}^{\oplus n}$-graded rings.
Since $\id_R(I/I^2)<\infty$, $R$ is Cohen--Macaulay, and thus $I$ is unmixed.
Let $\omega_R$ be the graded canonical module.
The natural homomorphism $\omega_R\otimes_R \Hom_R(\omega_R,I/I^2)\to I/I^2$ of $\Z^{\oplus n}$-graded modules is an isomorphism.
For minimal homogeneous systems of generators $\alpha_1,\ldots, \alpha_r$ and $\beta_1,\ldots,\beta_s$ of $\omega_R$ and $\Hom_R(\omega_R,I/I^2)$, respectively, one has
$r=\type(R)$ and $\{\alpha_i\otimes \beta_j \mid 1\le i\le r, 1\le j\le s\}$ is a minimal homogeneous system of generators $\omega_R\otimes_R \Hom_R(\omega_R,I/I^2)$.
In particular, $\mu_Q(I)=\mu_R(I/I^2)=rs$ holds.
For $\boldsymbol{a}_i=(a_{i1},\ldots,a_{in}):=\deg \alpha_i,\ \boldsymbol{b}_j=(b_{j1},\ldots,b_{jn}):=\deg \beta_j\in\Z^{\oplus n}$, we have $\deg(\alpha_i\otimes \beta_j)=\boldsymbol{a}_i+\boldsymbol{b}_j$ and hence $I$ is generated by the monomials 
$$
\{X^{\boldsymbol{a}_i+\boldsymbol{b}_j}:=\prod_{h=1}^n X_h^{a_{ih}+b_{jh}} \mid 1\le i\le r, 1\le j\le s\}.
$$
In particular, $a_{ih}+b_{jh}\ge 0$ for all $1\le i\le r$, $1\le j\le s$, and $1\le h\le n$.
If we set $c_h:=\min\{a_{ih}\mid 1\le i\le r\}$ and $\boldsymbol{c}=(c_1,\ldots,c_n)$, then $a_{ih}-c_h\ge 0$.
On the other hand, for each $h$, there exists some $i$ such that $a_{ih}=c_h$, which implies that $b_{jh}+c_h=b_{jh}+a_{ih}\ge 0$ for every $j$.
We see that $\boldsymbol{a}_i-\boldsymbol{c}, \boldsymbol{b}_j+\boldsymbol{c}\in\mathbb{N}^{\oplus n}$ and $I=JK$ for monomial ideals
\begin{equation}\label{J and K}
J=(X^{\boldsymbol{a}_i-\boldsymbol{c}} \mid 1\le i\le r) \quad \text{and} \quad K=(X^{\boldsymbol{b}_j+\boldsymbol{c}} \mid 1\le j\le s).
\end{equation}

First, we assume that $R$ is Artinian.
The assertion is clear when $n=1$. 
If $n=2$, then by the Hilbert--Burch theorem, we get $rs=\mu(I)=\type(R)+1=r+1$, and thus $r(s-1)=1$, which forces $r=1$. 
Therefore, we may assume that $n\ge 3$.
As discussed before, it suffices to show that $R$ is Gorenstein.
Suppose that $r\ge 2$ for contradiction.
By the graded Matlis duality, one has $\Hom_k (\operatorname{Soc}R, k)\cong \omega_R/\n \omega_R$ as graded $R$-modules.
Since $\boldsymbol{a}_1,\ldots,\boldsymbol{a}_r$ are the degrees of a minimal homogeneous system of generators of $\omega_R$, those for $\operatorname{Soc}R$ are their additive inverses, that is, 
\begin{equation}\label{generate Soc}
\operatorname{Soc}R=(X^{-\boldsymbol{a}_1},\ldots,X^{-\boldsymbol{a}_r})R.
\end{equation}

We fix $1\le h\le n$.
Since both $J$ and $K$ contain the $\n$-primary ideal $I$, we may define the following as nonnegative integers:
$$
p_h:=\min\{ p\ge 0 \mid X_h^p\in J\}\quad \text{and}\quad q_h:=\min\{ q\ge 0 \mid X_h^q\in K\}
$$
Since $r>1$, we have $\mu(I)=rs>s\ge \mu(K)$, and hence $I=JK\ne K$. 
So $J$ is a proper ideal, and $p_h$ is positive.
The fact that $X_h^{p_h}=X^{p_h\boldsymbol{e}_h}$ belongs to the monomial ideal $J$ implies that among the monomial generators of $J$ there is one that is a power of $X_h$, and it must divide $X_h^{p_h}$. 
By the minimality of $p_h$ and \eqref{J and K}, there is $1\le i\le r$ such that $\boldsymbol{a}_i-\boldsymbol{c}=p_h\boldsymbol{e}_h$.
By \eqref{generate Soc}, every component of $-\boldsymbol{a}_i$ is nonnegative, so
\begin{equation}\label{ineqp}
1\le p_h\le p_h-a_{ih}= -c_h.
\end{equation}
We show $X_h^{-c_h+1}\in I$.
Otherwise, $X_h^{-c_h+1} R\ne 0$.
If $X_h^{-c_h+1}R$ is not contained in $\operatorname{Soc}R$, then there exists a variable whose product with $X_h^{-c_h+1}R$ is nonzero. Repeating this process, we obtain $0\ne (X_h^{-c_h+1}f)R\subseteq \operatorname{Soc}R$ for some monomial $f$.
However, $-c_h=-\min\{a_{ih}\mid 1\le i\le r\}=\max\{-a_{ih}\mid 1\le i\le r\}$ is the largest exponent of $X_h$ among the monomials appearing in $\operatorname{Soc}R$.
This is a contradiction, and hence $X_h^{-c_h+1}\in I$.
We see that 
$$
p_h+q_h=\min\{t\ge 0 \mid X_h^t\in JK=I\}\le -c_h+1,
$$
and thus $q_h\le -c_h$ by \eqref{ineqp}.
Since $n\ge 3$, we may apply the above argument with $h=1,2,3$. 
As $p_1\le -c_1$ and $q_2\le -c_2$, the monomial $X_1^{p_1}X_2^{q_2}\in JK=I$ divides $X^{-\boldsymbol{c}-p_3\boldsymbol{e}_3}$, which implies that $X^{-\boldsymbol{c}-p_3\boldsymbol{e}_3}\in I$. 
Then the image of $X^{-\boldsymbol{c}-p_3\boldsymbol{e}_3}$ in $R$ is zero.
This contradicts the fact that the image of $X^{-\boldsymbol{c}-p_3\boldsymbol{e}_3}$ is one of the minimal monomial generators of $\operatorname{Soc}R$.

Next, we consider the general case.
Let $g=\height(I)$.
The case $g=0$ is trivial, while if $g=1$, every unmixed monomial ideal of height $1$ is principal, so the conclusion follows.
We may assume $g\ge 2$.
We see that $R$ is generically a complete intersection.
Indeed, since $I$ is a monomial ideal, every minimal prime $\p$ of $I$ is of the form $\p=(X_1,\ldots, X_g)$ after a suitable permutation of the variables.
Hence we have
$$
Q_\p\cong L[X_1,\ldots, X_g]_{(X_1,\ldots, X_g)}, \ \ \text{where} \ \ L=k(X_{g+1},\ldots, X_n).
$$
By the Artinian case, $Q_\p/IQ_\p$ is a complete intersection, and thus $IQ_\p=(X_1^{u_1},\ldots, X_g^{u_g})Q_\p$ for some positive integers $u_1,\ldots, u_g$.
Moreover, in this case, either $JQ_\p$ or $KQ_\p$ is the unit ideal by \cite[Proposition 1]{H3}, and hence $(J\cap K)_\p=I_\p=(JK)_\p$.

Since $R$ is generically a complete intersection, $\Hom_R(\omega_R,I/I^2)$ has rank $g$, and thus  
$$
s=\mu(\Hom_R(\omega_R,I/I^2))\ge g\ge 2.
$$
Here again, it suffices to assume that $r\ge 2$ and derive a contradiction.
Then since both $r$ and $s$ are greater than $1$, each is strictly less than $rs$.
By the same argument as in the Artinian case, this implies that both $J$ and $K$ are proper.
The previous paragraph shows that $\Supp((J\cap K)/JK)\cap\Ass(R)=\emptyset$.
On the other hand, since $(J\cap K)/JK$ is an ideal of $R$, we have $\Ass((J\cap K)/JK)\subseteq \Ass(R)$, which shows that 
$$
\Tor_1^Q(Q/J,Q/K)\cong (J\cap K)/JK=0.
$$
Then $\Tor_i^Q(Q/J,Q/K)=0$ for all $i>0$ by the rigidity for regular rings; see \cite[Corollary 1]{L}.
Therefore, $J\otimes_Q K\cong JK$ by the exact sequence
$$
0=\Tor_2^Q(Q/J,Q/K)\cong \Tor_1^Q(J,Q/K) \to J\otimes_Q K \to J \to J\otimes_Q (Q/K)\cong J/JK \to 0.
$$
If $F$ and $G$ are minimal free resolutions of $J$ and $K$ over $Q$, respectively, then $F\otimes_Q G$ is a minimal free resolution of $J\otimes_Q K\cong JK=I$. Consequently, $\pd_Q J+\pd_Q K=\pd_Q I$ and thus
$$
\pd_Q (Q/J)+\pd_Q (Q/K)=(\pd_Q J+1)+(\pd_Q K+1)=\pd_Q J+\pd_Q K+2=\pd_Q I+2=\pd_Q (Q/I)+1.
$$
In general, for any homogeneous ideal $H$, the Auslander--Buchsbaum formula gives 
$$
\pd_Q(Q/H)=n-\depth(Q/H)\ge n-\dim(Q/H)=\height(H).
$$
As $R=Q/I$ is Cohen--Macaulay, it follows that
$$
2g\le \height(J)+\height(K)\le \pd_Q (Q/J)+\pd_Q (Q/K)=\pd_Q (Q/I)+1=\height(I)+1=g+1.
$$
This contradicts $g\ge 2$. This completes the proof.
\end{proof}

\begin{cor}\label{base change}
Let $Q=S[X_1,\ldots,X_n]$ be a polynomial ring over a reduced ring $S$,  let $I$ be a monomial ideal of $Q$, and $R=Q/I$. 
If $\id_R(I/I^2)<\infty$, then $I$ is a complete intersection.
\end{cor}

\begin{proof}
Let $\p$ be a minimal prime ideal of $S$. Then $Q_\p=S_\p[X_1,\ldots,X_n]$ is a polynomial ring over a field $S_\p$, and $IQ_\p/(IQ_\p)^2\cong(I/I^2)_\p$ has finite injective dimension. By Theorem \ref{monomial over k}, $IQ_\p$ is generated by a regular sequence. Since $IQ_\p$ is a monomial ideal of $Q_\p$, no variable occurs in two distinct minimal generators; see \cite[Remark 3.3.7]{OW} for instance.
Since $I$ and $IQ_\p$ are generated by the same monomials, this property also holds for $I$. Hence the minimal generators of $I$ clearly form a $Q$-regular sequence.
\end{proof}

\begin{cor}
Let $(Q,\n,k)$ be a Noetherian local ring, let $x_1,\ldots,x_n$ be a $Q$-regular sequence, let $I$ be a nonzero proper ideal of $Q$ generated by monomials in  $x_1,\ldots,x_n$, and set $R=Q/I$.
If $\id_R(I/I^2)<\infty$, then $Q$ and $R$ are Gorenstein and $I$ is a complete intersection.
\end{cor}

\begin{proof}
Let $B=Q[X_1,\ldots,X_n]_{(\n, X_1,\ldots,X_n)}$ be a local ring, let $J\subseteq B$ be the monomial ideal corresponding to $I$, and set $A=B/J$ and $C=J/J^2$.
For every monomial ideal $H$ of $Q[X_1,\ldots,X_n]$, we show inductively that $X_1-x_1,\ldots, X_n-x_n$ is a regular sequence over $Q[X_1,\ldots,X_n]/H$.
After factoring out $X_1-x_1,\ldots, X_{i-1}-x_{i-1}$, the quotient ring is $Q[X_i,\ldots,X_n]/H_i$, where $H_i$ is a monomial ideal in $x_1,\ldots,x_{i-1},X_i,\ldots,X_n$. 
Since none of its monomial generators involves $x_i$, applying \cite[Corollary 1]{KS} to the permutable regular sequence $x_1,\ldots,x_i,X_i,\ldots,X_n$ on $Q[X_i,\ldots,X_n]$ gives $H_i:x_i=H_i$.
Thus $x_i$ is regular on $Q[X_i,\ldots,X_n]/H_i$. Giving $X_i$ degree $1$ and all other elements degree $0$, and comparing the lowest-degree terms shows immediately that $X_i-x_i$ is regular on $Q[X_i,\ldots,X_n]/H_i$.
In particular, $X_1-x_1,\ldots, X_n-x_n$ is a regular sequence over $A$ and $B/J^2$. 
Hence it is also a regular sequence on $C$ by the exact sequence $0\to C\to B/J^2\to A\to 0$. One has 
$$
A/(X_1-x_1,\ldots, X_n-x_n)\cong Q/I=R, \quad
C\otimes_A (A/(X_1-x_1,\ldots, X_n-x_n))\cong I/I^2.
$$
Hence $\id_A(C)<\infty$. Every injective $A$-module $E$ satisfies $\Hom_Q(-,E)\cong \Hom_A(A\otimes_Q(-), E)$. As $A$ is flat over $Q$, this functor is exact, which means that $E$ is injective over $Q$. By this, we have $\id_Q(C)\le\id_A(C)<\infty$.
On the other hand, since $A$ and $B/J^2$ are flat over $Q$, so is $C$.
Moreover, $C$ is faithfully flat over $Q$. Indeed, since $C$ is a nonzero finitely generated $A$-module, $C/\n C\ne 0$ by Nakayama's lemma.
We have
$$
\rhom_Q(k,C)\simeq \rhom_Q(k,Q)\otimes_Q^\mathbf{L} C=\rhom_Q(k,Q)\otimes_Q C
$$
and it is bounded as $\id_Q(C)<\infty$. Since $C$ is faithfully flat, $\rhom_Q(k,Q)$ is also bounded, which implies that $Q$ is Gorenstein.
We see that 
$$
\rhom_A(A/\n A, C)\simeq\rhom_A(k\otimes_Q^{\mathbf{L}} A, C)\simeq \rhom_Q(k,C)\simeq \rhom_Q(k,Q)\otimes_Q C\simeq k[-d]\otimes_Q C,
$$
where $d=\dim Q$, and 
$$
\rhom_A(k, C)\simeq\rhom_{A/\n A}(k, \rhom_A(A/\n A, C))\simeq\rhom_{A/\n A}(k, C/\n C)[-d],
$$
which is bounded. 
Let $K$ be the monomial ideal of $k[X_1,\ldots,X_n]$ corresponding to $J$.
Then $C/\n C \cong (K/K^2)_{(X_1,\ldots,X_n)}$ has finite injective dimension over $A/\n A\cong (k[X_1,\ldots,X_n]/K)_{(X_1,\ldots,X_n)}$.
Viewing this in terms of localization at the homogeneous maximal ideal of the positively graded ring, this is equivalent to saying that $K/K^2$ has finite injective dimension over $k[X_1,\ldots,X_n]/K$.
By Theorem \ref{monomial over k}, $K$ is generated by a regular sequence. Since $K$ is a monomial ideal, no variable occurs in two distinct minimal generators; see \cite[Remark 3.3.7]{OW} for instance.
This property also holds for $I$. 
As $x_1,\ldots,x_n$ is a $Q$-regular sequence, the minimal generators of $I$ form a $Q$-regular sequence; see \cite[Lemma 2]{KS} for instance.
Since $Q$ is Gorenstein, so is $R$.
\end{proof}

\begin{ac}
Part of this paper was done during Kimura's visit to the Florida A\&M University in June, 2026. He is grateful to the Department of Mathematics for their hospitality and to Ryo Takahashi for supporting his travel expenses for the visit.
He also thanks Siva Somasundaram for sharing helpful materials on linkage theory.
The authors are indebted to Mohsen Asgharzadeh for proposing Corollary \ref{base change} and for kindly agreeing to its inclusion in this paper.
ChatGPT (OpenAI ChatGPT 5.6 Sol in September 2026) assisted the authors in exploring ideas and organizing some of the arguments during the preparation of this paper.
\end{ac}


\end{document}